\documentclass[a4paper,11pt]{amsart}

\usepackage{amsmath, amssymb, amsthm}
\usepackage{graphicx}
\usepackage{subfigure}
\usepackage{color}
\usepackage[hyperindex,breaklinks]{hyperref}
\usepackage{xspace}
\usepackage{framed}
\usepackage{algorithm}
\usepackage[noend]{algpseudocode}

\usepackage[a4paper]{geometry} \hypersetup{colorlinks = true, urlcolor = blue, linkcolor = blue,
  citecolor = blue }

\newtheorem{theorem}{Theorem}

\newtheorem{lemma}{Lemma}

\theoremstyle{remark}

\newcommand{\ve}{{\varepsilon}}
\newcommand{\E}{{\mathbb{E}}}

\newcommand{\xmath}[1]{\ensuremath{#1}\xspace}
\renewcommand{\Pr}{\xmath{\mathbb{P}}}

\newcommand{\Uup}{U^\uparrow}

\newcommand{\Ldown}{L^\downarrow}
\newcommand{\Ijn}{\mathcal{I}_{j,n}}
\newcommand{\Tl}{T_{left}}
\newcommand{\Tr}{T_{right}}

\begin{document}

\title[Exact Simulation of Multidimensional RBM]{Exact Simulation of
  Multidimensional\\ Reflected Brownian Motion} 
\author[Blanchet and Murthy]{{\large
    J\MakeLowercase{ose} B\MakeLowercase{lanchet}} \hspace{32pt}
  {\large K\MakeLowercase{arthyek} M\MakeLowercase{urthy}}\\\\
  \textit{\large S\MakeLowercase{tanford} U\MakeLowercase{nivesity}
    \hspace{20pt} C\MakeLowercase{olumbia}
    U\MakeLowercase{niversity}}} 

\address{Stanford University, Management Science \& Engineering
  Department, 475 Via Ortega, Stanford, CA 94305-4121, United States.} 
\email{jblanche@stanford.edu} 
\address{Columbia University, Department of Industrial
  Engineering \& Operations Research, 340 S. W. Mudd Building, 500
  W. 120 Street, New York, NY 10027, United States.}
\email{karthyek.murthy@columbia.edu} 

\thanks{Author 1 gratefully acknowledges the support from NSF Award
  1538217.} 


\date{Dec 15, 2016} 

\maketitle

\begin{abstract}
 We present the first exact simulation method for multidimensional
  reflected Brownian motion (RBM). Exact simulation in this setting is
  challenging because of the presence of correlated local-time-like
  terms in the definition of RBM. We apply recently developed
  so-called $\varepsilon $-strong simulation techniques (also known as
  Tolerance-Enforced Simulation) which allow us to provide a
  piece-wise linear approximation to RBM with $\varepsilon $
  (deterministic)\ error in uniform norm. A novel conditional
  acceptance / rejection step is then used to eliminate the error. In
  particular, we condition on a suitably designed information
  structure so that a feasible proposal distribution can be applied.
\end{abstract}

\section{Introduction}
\label{Sec-Intro}
This paper is a contribution to the theory of exact simulation for
stochastic differential equations (SDEs). In particular, we present
the first exact simulation algorithm for multidimensional reflected
Brownian motion (RBM).

Multidimensional RBM was introduced by Harrison and Reiman in
\cite{MR606992} and it figures prominently in stochastic Operations
Research. It turns out that RBM approximates the workload at each
station in so-called generalized Jackson networks, which are comprised
of single-server queues connected via Markovian routing. The
approximation holds in heavy traffic (that is, as the system
approaches 100\% utilization) and it is applicable in great generality
(assuming only a functional central limit theorem for the arrival
process and the service requirements at each station, see for example
\cite{Reiman_1984} and \cite{chen2001fundamentals}). Following
\cite{MR606992}, we refer a $d$-dimensional stochastic process
$(\mathbf{Y}(t) : t \geq 0)$ that satisfies the following properties
as a reflected Brownian motion (RBM):
\begin{itemize}
\item[1)] $\mathbf{Y}(\cdot)$ is a Markov process with stationary
  transition probabilities, continuous sample paths taking values in
  the non-negative orthant of $\mathbb{R}^d,$
\item[2)] $\mathbf{Y}(\cdot)$ behaves in the interior of positive orthant like
  a $d$-dimensional Brownian motion (either standard or with a
  constant drift and diffusion matrix),
\item[3)] $\mathbf{Y}(\cdot)$ reflects instantaneously at the boundary of the
positive orthant, and
\item[4)] the direction of reflection anywhere on the boundary surface
  where the $i$-th component $Y_i = 0$ is the $i$-th column of the
  $d \times d$ {\textit reflection matrix} $R.$ It is required that
  $R$ is of the form $R = I - Q^T,$ where $Q$ is a non-negative
  $d \times d$ matrix with zeros on the diagonal and spectral radius
  strictly smaller than unity.
\end{itemize}

\noindent
\textbf{RBM as a solution of Skorokhod problem.} Let
$(\mathbf{X}(t): t \geq 0)$ denote a $d$-dimensional Brownian motion
that behaves similar to RBM $\mathbf{Y}(\cdot)$ in the interior of the
positive orthant. Then it is well known that the RBM
$\mathbf{Y}(\cdot),$ defined above, can be represented as
\begin{align}
  \mathbf{Y}(t) = \mathbf{X}(t) + R\mathbf{L}(t),
  \label{SKO-PROB}
\end{align}
with $\mathbf{Y}(t)=(Y_{1}(t),\ldots ,Y_{d}(t))^{T}$ and $\mathbf{L}%
(t)=(L_{1}(t),\ldots ,L_{d}(t))^{T}$ satisfying,
\begin{itemize}
\item[1)] $Y_{i}(t)\geq 0\text{ for all }t\geq 0$,
\item[2)] $L_{i}(t)$ is non-decreasing in $t$, and $L_{i}(0)=0$
\item[3)] $\int_{0}^{t}Y_{i}(s){d}L_{i}(s)=0$, 
\end{itemize}
for each $t\geq 0$ and $i=1,\ldots ,d$ (see, for example,
\cite{MR606992,chen2001fundamentals}). We call $\mathbf{X}(\cdot)$ the
driving (or free) process, and $\mathbf{Y}(\cdot)$ the reflected
process. The map $S$ that takes $\mathbf{X}(\cdot)$ to
$\mathbf{Y}(\cdot)$ in (\ref{SKO-PROB}) is referred to as the
\textit{Skorokhod map.} Item 3) above simply states that the process
$L_i(t)$ increases only at those times $t$ where $Y_i(t) = 0.$ Because
of this property, the process $L_i(\cdot)$ behaves like the local time
of Brownian motion at the origin. Consequently, the term
$R\mathbf{L}(t)$ appearing in (\ref{SKO-PROB}) is not a standard
``drift'' term, and cannot be dealt with using change of measure
techniques as in \cite{beskos2005}, \cite{MR2274855} or
\cite{pollock2016}.

All the generic exact simulation techniques for diffusions are based
on the acceptance / rejection, after applying Girsanov's
transformation. The difficulty in applying acceptance / rejection in
the multidimensional RBM setting is that there is no natural proposal
distribution that can be used to \textquotedblleft
dominate\textquotedblright\ the target process directly.  In
particular, multidimensional RBM is not absolutely continuous with
respect to any natural process that is easily simulatable. Note that
in one dimension one can simulate RBM directly by keeping track of the
running maximum of the driving Brownian motion, 
and so these challenging issues arise only in dimensions greater than
one. Simulation techniques for one dimensional reflected processes
have been studied in \cite{etore2013exact}.
\newline 

\noindent
\textbf{Our contributions.}  This paper is dedicated to the proof of
the following result. Let $\mathbf{Y}(\cdot)$ denote the
multi-dimensional RBM in (\ref{SKO-PROB}).
\begin{theorem}
  Given a deterministic time $T \in (0,1),$ it is possible to simulate
  $\mathbf{Y}\left( T\right) $ without any bias.
\end{theorem}

An obstacle to naively using the traditional acceptance / rejection
algorithm (see, for example, \cite{Asmussen&Glynn_2007}) in the
simulation of diffusions is that the probability density from which we
want to sample is typically unknown. In our setting, while the
probability density of $\mathbf{Y}(T)$ itself may be unknown, we
propose to simulate enough information about the RBM
$\mathbf{Y}(\cdot),$ and perform acceptance / rejection sampling for
the probability density of $\mathbf{Y}(T)$ conditional on the
simulated filtration. Ideally, the simulated information set should be
a collection of random variables such that the probability density of
$\mathbf{Y}(T)$ conditional on the simulated information, denoted here
by $f,$ is computable. If obtaining such a computable conditional
density $f$ is feasible, then one can easily perform an acceptance /
rejection step of form,
\begin{align}
  V < \frac{f(\mathbf{Z})}{Cg(\mathbf{Z})} =:
  L(\mathbf{Z}), 
\label{BASIC-AR}
\end{align}
where $g$ is a suitable proposal density from which proposal samples
$\mathbf{Z}$ are drawn, $V$ is an independently generated random
variable distributed uniformly in $[0,1],$ and $C$ is a suitable
scaling constant. While this line of thought is interesting, a key
difficulty arises from the fact that there is no easily simulatable
information structure such that the density of $\mathbf{Y}(T)$
conditional on the simulated information is exactly computable.

To overcome this difficulty, we introduce a novel sampling scheme that
we call as \textit{Refine until Accept / Reject}, which relaxes the
requirement that the conditional density $f$ is known exactly. The key
observation behind this algorithm is that in order to accept the
proposal $\mathbf{Z},$ we simply need to decide if inequality
(\ref{BASIC-AR}) holds; \textit{we do not need to know the right-hand
  side of (\ref{BASIC-AR}) exactly.} So, instead of having direct
access to the probability density of $\mathbf{Y}$ conditional on
$\mathcal{I},$ if we can simply obtain an approximation to the
right-hand side of (\ref{BASIC-AR}) that ensures inequality
(\ref{BASIC-AR}) holds, we can accept the proposed sample without
incurring any sampling error. We present this idea clearly in a
stylized setting in Section \ref{Sec-RAR-Intro} along with an outline
of its applicability to the simulation of RBM in Section
\ref{Sec-RAR-RBM}. We use Section \ref{SEC-SAMP-PROC} to fully present
our algorithm for exactly simulating multidimensional RBM. Our
algorithm relies on the recently developed $\varepsilon$-strong
simulation (also known as Tolerance-Enforced simulation) techniques in
\cite{MR2995793} to first derive an approximation of the RBM, which
then is used to make one of the following decisions: Accept, reject,
(or) refine the approximation of the right-hand side in
(\ref{BASIC-AR}) until either the proposal can be conclusively
accepted or rejected.

We wish to finish this introduction with a critical discussion of our
main result. We do believe that the conditional acceptance / rejection
strategy introduced here is of significant value as it addresses an
important open problem (exact sampling of multidimensional
RBM). Nevertheless, we must recognize that the algorithm, in its
current form, is mostly of theoretical interest. Unfortunately, in
Section \ref{Sec-Complexity}, we identify that the expected running
time of the algorithm is infinite. While we are investigating
strategies to mitigate this problem, we feel that the nucleus of our
sampling algorithm, namely refine until accept / reject, might propel
further research in exact sampling of various stochastic processes in
addition to the search for efficient sampling algorithms for
simulating multidimensional RBM.

\section{Overview of the sampling scheme}
\label{SEC-SAMP-PROC}
We first introduce some notational conventions.  Throughout the paper
we consider the driving (free) stochastic process $\mathbf{X}(\cdot )$
to be a standard Brownian motion in $d$-dimensions, which write as
$\mathbf{X}(\cdot) =\mathbf{B}(\cdot ).$ The reflected process
$\mathbf{Y}(\cdot )$ in \eqref{SKO-PROB} is referred to as the
Reflected Brownian motion (RBM). The extension of our development to
the case in which $\mathbf{X}(\cdot) $ is a Brownian motion with
constant drift and diffusion coefficients is straightforward.  As
mentioned in the Introduction, the map $S$ that takes
$\mathbf{X}(\cdot)$ to $\mathbf{Y}(\cdot)$ in (\ref{SKO-PROB}) is
referred to as the Skorokhod reflection map.

While all the variables and stochastic processes taking values in
$\mathbb{R}^{d}$ for $d>1$ are typeset in boldface, their
1-dimensional counterparts are not. For example, if
$\mathbf{B}(\cdot )$ denotes the Brownian motion in multiple
dimensions, then $B(\cdot )$ is to be understood as 1-dimensional
Brownian motion.

\subsection{Refine until accept / reject sampling scheme: An introduction}
\label{Sec-RAR-Intro}
In this section, let us restrict our attention to the following
sampling problem to which our RBM simulation problem is later shown to
be reduced: Let $\Delta $ and $Y$ be two independent random variables,
and $W = Y + \Delta.$ For simplicity, let us assume that the
probability density of $\Delta $, denoted by
$f_{\Delta }\left( \cdot \right) $, is continuous on its support which
is given by the interval $[-a,a]$ for some $a >0.$ Consecutively,
$\sup_{x\in \lbrack -a,a]}f_{\Delta }\left( x\right) \leq C$ for some
$C\in \left( 0,\infty \right) $. Let us assume that $Y$ is also
supported on $[-a,a]$ with an arbitrary distribution from which we do
not know how to sample exactly. Our objective in this section is to
obtain samples from the distribution of $W = Y + \Delta.$

Given $Y=y$, the density of $W$ evaluated at $w$ is simply
$f_{\Delta }\left( w-y\right) $.  Moreover, the support of such
density is contained in the interval $[-2a,2a]$. So, in order to
simulate $W$ conditional on $y$, we could propose $W$ uniformly
distributed in $[-2a,2a]$, and accept if
\begin{equation}
  V<C^{-1}f_{\Delta }\left( W-y\right) ,  \label{INEQ_AR}
\end{equation}%
where $V\sim \text{Unif}\left( 0,1\right) $ (uniform in $(0,1)$), and
independent of $%
W $ and $Y=y$.

The key observation is that in order to accept $W$ we simply need to
decide if inequality (\ref{INEQ_AR}) holds; \textit{we do not actually
  need to know the value of }$y$. So, instead of having direct access
to $Y$, there are settings, as we demonstrate in our RBM simulation,
where we might know $\left\{ Y^{\varepsilon _{n}}\right\} _{n\geq 1}$,
independent of $\Delta $, which converges to $Y$; say
$\left\Vert Y^{\varepsilon _{n}}-Y\right\Vert \leq \varepsilon _{n}$
for some $\varepsilon _{n}\rightarrow 0$ as $n\rightarrow \infty $.
Then under modest continuity properties of
$f_{\Delta }\left( \cdot \right) $, for instance say $%
\left\vert f_{\Delta }\left( x\right) -f_{\Delta }\left( x^{\prime
    }\right) \right\vert \leq K\left\Vert x-x^{\prime }\right\Vert $,
we can accept $W$ if
\begin{equation}
  V\leq C^{-1}f_{\Delta }\left( W-Y^{\varepsilon _{n}}\right)
  -KC^{-1}\varepsilon _{n},  \label{I1_INT}
\end{equation}%
or reject $W$ if 
\begin{equation}
  V\geq C^{-1}f_{\Delta }\left( W-Y^{\varepsilon _{n}}\right)
  + KC^{-1}\varepsilon _{n}.  \label{I2_INT}
\end{equation}
Since $\varepsilon _{n}\rightarrow 0$ and
$V=C^{-1}f_{\Delta }\left( W-Y\right) $ has zero probability of
occurring, one must be able to eventually decide whether to accept or
reject. As the outlined sampling procedure does one of the following
-- accept / reject the proposed sample of $W,$ or seek for a
refinement of $Y^{\varepsilon_n}$ until the proposal can be accepted
or rejected -- we call the sampling procedure as \textit{refine until
  accept / reject}.

It is useful to remember the following requirements which are
necessary for the procedure underlying (\ref{INEQ_AR}), (\ref{I1_INT})
and (\ref{I2_INT}) to yield exact samples of $W = Y + \Delta:$
\begin{itemize}
\item[R1)] The probability density of $\Delta,$ denoted by
  $f_{\Delta}(\cdot),$ has bounded support, and is Lipschitz
  continuous; that is, there exists $K > 0$ such that
  $\vert f_\Delta(x) - f_\Delta(x') \vert \leq K \Vert x - x' \Vert$
  for all $x,x'.$
\item[R2)] Conditional on $Y = y$ and all the information simulated to
  obtain $Y^{\varepsilon_n},$ the probability density of $W$ evaluated
  at $w$ is simply $f_{\Delta}(w-y).$
\end{itemize}

\subsection{An outline of the application of refine until accept /
  reject for multi-dimensional RBM}
\label{Sec-RAR-RBM}
Revisiting our objective of exact sampling of RBM, our plan is to
apply the sampling strategy in Section \ref{Sec-RAR-Intro} by
introducing a suitable conditioning. For this purpose, we use the
following key facts about multidimensional RBM. First, the fact that
if the driving process is Brownian motion then, for fixed $T$,
$\Pr(Y_{i}\left( T\right) =0)=0$ for any $%
i\in \{1,...,d\}$.
In addition, since $\mathbf{Y}\left( \cdot \right) $ is continuous,
there exists a $\delta >0$ and an interval $%
(T_{left},T_{right}]$
which contains $T$, satisfying $Y_{i}\left( s\right) >\delta$ for all
$i\in \{1,...,d\},$ and 
therefore,
\begin{align*}
  \mathbf{Y}(s)=\mathbf{Y}(T_{left})+\mathbf{X}(s)-\mathbf{X}(T_{left}),
\end{align*}
for all $s\in
(T_{left},T_{right}].$ In other words, the interval $(\Tl,
\Tr)$ is such that the RBM
$\mathbf{Y}(t)$ does not hit the reflecting boundary anywhere during $t
\in (\Tl,\Tr]$ and consecutively, $\mathbf{L}(s) - \mathbf{L}(\Tl) =
0$ for all $s \in (\Tl, \Tr).$

So, our plan is to first simulate enough information about
$\mathbf{X}\left( \cdot \right) $ (that is, the driving Brownian
motion) so that conditional on such information we have the
representation
\begin{align}
  \mathbf{Y}(T) &=\mathbf{Y}(T_{left})+\big(\mathbf{X}(T)-\mathbf{X}(T_{left})\big) \nonumber\\
&=:\mathbf{Y}(T_{left})+ \mathbf{\Delta},  \label{EQ_L_ZERO}
\end{align}
for a suitable $T_{left}$ identified from the simulated
information. Naturally, we identify $\mathbf{Y}(T)$ and
$\mathbf{Y}(\Tl),$ respectively, with the variables $W$ and $Y$
introduced in the abstract setting discussed previously in Section
\ref{Sec-RAR-Intro}. Therefore, our objective is to simulate just
enough information so that conditioned on the simulated information,
\begin{itemize}
\item[R1')] the probability density of the Brownian increment 
  $\mathbf{\Delta} := \mathbf{X}(T) - \mathbf{X}(\Tl),$
  denoted by $f_{\mathbf{\Delta}}(\cdot),$ has bounded support, and is
  lipschitz continuous, and
\item[R2')] the probability density of $\mathbf{Y}(T),$ evaluated at
  $w,$ is simply $f_{\mathbf{\Delta}}(w-\mathbf{Y}(\Tl)).$
\end{itemize}
The requirements R1') and R2') mirror the earlier requirements R1) and
R2) in the abstract setting in Section \ref{Sec-RAR-RBM}. Once these
requirements are met, we can follow the logic in Section
\ref{Sec-RAR-Intro} to develop a refine until accept / reject sampler
for obtaining samples from the distribution of $\mathbf{Y}(T)$. Thus,
the proposed algorithm for exact sampling of $\mathbf{Y}(T)$ can be
roughly divided into two steps:
\begin{itemize}
\item[1)] a preconditioning step where we simulate enough information to
  arrive at the representation (\ref{EQ_L_ZERO}), and 
\item[2)] exploiting the representation (\ref{EQ_L_ZERO}) arrived in
  the preconditioning step, we perform `refine until accept / reject' to
  obtain samples of $\mathbf{Y}(T).$
\end{itemize}

\subsubsection{Overview of the preconditioning step.}
In order to sample enough information which will enable us to obtain
the representation (\ref{EQ_L_ZERO}) along with satisfying the above
two requirements, we use another important property of the Skorokhod
map, $S$, namely, $S$ is Lipschitz continuous as a function of the
driving process in the uniform norm over the time interval
$[0,1]$. Consequently, to identify $T_{left}$ we use so-called
$\varepsilon-$strong simulation techniques, also known as
Tolerance-Enforced Simulation (TES), which allows us to simulate
$\mathbf{X}^{\varepsilon }(\cdot) $ piecewise linear and guaranteed to
be within $\varepsilon$-close in uniform norm to $\mathbf{X}(\cdot) $.
This construction is, conceptually at least, not complicated. There
are several methods that can be applied for the same: based on
wavelets as in \cite{blanchet2012steady}, localization using stopping
times as in \cite{chen2013localization}, or tracking jointly the
maxima and minima on dyadic intervals as in \cite{MR2995793}. We have
chosen to use the latter construction, thereby ultimately obtaining
$(T_{left},T_{right}]$ as a dyadic interval (i.e. $T_{left}=i2^{-N}$
and $T_{right}=j2^{-N}$ for some $0\leq i<j\leq 2^{N}$ and $N>0$). The
reason for choosing the construction in \cite{MR2995793} is because it
allows us to recursively develop more refined approximations
$\mathbf{X}^{\varepsilon'}$ for desired $\varepsilon' < \varepsilon$
while preserving the conditional independence of $\mathbf{\Delta}$ and
$\mathbf{Y}( T_{left})$ given all the information required to conclude
that $\mathbf{L}( T_{right}) -\mathbf{L}( T_{left}) =0.$ 
Refer Section \ref{SEC-STR-APPROX} for an overview of the algorithm in
\cite{MR2995793} that allows us to obtain the desired piecewise linear
approximation $\mathbf{X}^\varepsilon(\cdot)$ for the driving Brownian
motion.


The Skorokhod problem is easy to solve for piecewise linear input
$\mathbf{X}^{\varepsilon }$, because in such case the solution to
Skorokhod problem, denoted by
$(\mathbf{Y}^{\varepsilon },\mathbf{L}^{\varepsilon })$ is piecewise
linear as well, and the gradients can be obtained by solving linear
systems based on (\ref{SKO-PROB}) (see Section
\ref{Sec-Str-Approx-RBM} for an explicit algorithm). Since the
piecewise linear approximation $\mathbf{Y}^{\varepsilon _{n}}(\cdot)$
to RBM $\mathbf{Y}(\cdot)$ can be identified explicitly for a
computable $\{\varepsilon _{n}\}_{n\geq 1}$ such that
$\varepsilon _{n}\rightarrow 0$ as $n\rightarrow \infty $, the
Lipschitz continuity of $\mathbf{Y} = S(\mathbf{X})$ as a function of
$\mathbf{X}$, combined with the approximation $%
\mathbf{X}^{\varepsilon _{n}}$,
and the fact that $\mathbf{Y}\left( T\right) $ must be strictly
positive coordinate-wise, eventually can be used to identify
$T_{left}$ used in the additive representation (\ref{EQ_L_ZERO}). See
Section \ref{Sec-Precondition} for details.

\subsubsection{Requirements for refine until accept / reject step.}
Once we arrive at representation (\ref{EQ_L_ZERO}), we can use the
refine until accept / reject algorithm introduced in Section
\ref{Sec-RAR-Intro} to obtain samples of $\mathbf{Y}(T).$ In order to
be able to do this, we need to guarantee that the requirements R1')
and R2') mentioned earlier are met. 
Our construction of $\mathbf{X}^{\varepsilon }\left( \cdot \right) $,
as indicated earlier, based on \cite{MR2995793} will give rise to a
conditional density for $\mathbf{\Delta},$ denoted by
$f_{\mathbf{\Delta}}(\cdot),$ which is expressed as an infinite
series. So, the Lipschitz continuity of $f_{\mathbf{\Delta}}(\cdot) $
used in (\ref{I1_INT}) and (\ref{I2_INT}) is obtained by means of some
careful estimates.  
Consequently, as we shall see in Section \ref{Sec-Cond-Dens-Delta}, we
will be able to implement the basic refine until accept / reject
strategy underlying (\ref{INEQ_AR}), (\ref{I1_INT}), and
(\ref{I2_INT}).

In Section \ref{SEC-SAMP-PROC} below, we provide more specific details
behind our sampling methodology and point to future relevant sections
where details are fully fleshed out.

\section{The exact sampling scheme}
\label{SEC-SAMP-PROC}
We first describe essential components of the pre-conditioning step
(such as $\varepsilon-$strong approximation techniques) before
presenting the refine until accept / reject sampler for
$\mathbf{Y}(T).$

\subsection{The preconditioning step}
\label{Sec-Precondition}
As mentioned in Section \ref{Sec-RAR-RBM}, the objective of the
preconditioning step is to simulate just enough information in order
to obtain the representation (\ref{EQ_L_ZERO}) while satisfying
requirements R1') and R2'). 

\subsubsection{Generating $\protect\varepsilon $-strong approximation
  for the driving Brownian motion.}
\label{SEC-STR-APPROX} 
Here, we first provide a brief description of the $\varepsilon$-strong
algorithm of \cite{MR2995793} that simulates a piecewise linear
approximation to 1-dimensional standard Brownian motion $B(\cdot).$
The algorithm iteratively generates a sequence of pairs of piecewise
constant dominating processes, $\{B^{\uparrow}_{n}(t):t \in[0,1]\}$
and $\{B^{\downarrow}_{n}(t):t \in[0,1]\},$
that satisfy the following properties: For all $t \in[0,1],$
\begin{equation*}
B^{\downarrow}_{n}(t) \leq B^{\downarrow}_{n+1}(t) \leq B(u) \leq
B^{\uparrow }_{n+1}(t) \leq B^{\uparrow}_{n}(t), \text{ and }
\end{equation*}
\begin{equation*}
\sup_{t \in[0,1]}|B^{\uparrow}_{n}(t) - B^{\downarrow}_{n}(t)| \searrow0, 
\text{ a.s. as } n \nearrow\infty.
\end{equation*}

At every step $n \geq1,$ the algorithm generates information about the
Brownian motion $B(\cdot)$ in dyadic intervals $\{((j-1)2^{-n},
j2^{-n}]:j=1,\ldots,2^{n}\}$ conditional on the information available
on dyadic intervals from the $(n-1)^{th}$ step. Let $m_{j,n}$ and
$M_{j,n}$ denote the extrema of $B(\cdot):$
\begin{equation*}
m_{j,n} := \inf\{B(t): t \in((j-1)2^{-n}, j2^{-n}]\} \text{ and } M_{j,n} =
\sup\{B(t): t \in((j-1)2^{-n}, j2^{-n}]\}.
\end{equation*}
During $n^{th}$ iteration, the $\varepsilon$-strong algorithm simulates the
following random quantities for each dyadic interval (indexed by $%
j=1,\ldots,2^{n}$):

\begin{enumerate}
\item[1)] an interval that contains the minimum: $L_{j,n}^{\downarrow
  }$ and $ L_{j,n}^{\uparrow }$ such that $m_{j,n}\in \lbrack
  L_{j,n}^{\downarrow },L_{j,n}^{\uparrow }]$ and $L_{j,n}^{\uparrow
  }-L_{j,n}^{\downarrow }<2^{-(n+1)/2},$

\item[2)] an interval that contains the maximum: $U_{j,n}^{\downarrow
  }$ and $ U_{j,n}^{\uparrow }$ such that $M_{j,n}\in \lbrack
  U_{j,n}^{\downarrow },U_{j,n}^{\uparrow }]$ and $U_{j,n}^{\uparrow
  }-U_{j,n}^{\downarrow }<2^{-(n+1)/2},$ and

\item[3)] the end-points of Brownian motion: $B((j-1)2^{-n})$ and
  $B(j2^{-n}).$
\end{enumerate}

Let $\mathcal{I}_{j,n}$ denote the collective information, 
\begin{equation*}
\mathcal{I}_{j,n} := \{L^{\downarrow}_{j,n}, L^{\uparrow}_{j,n},
U^{\downarrow}_{j,n}, U^{\uparrow}_{j,n}, B((j-1)2^{-n}), B(j2^{-n})\},
\end{equation*}
which is referred to as \textit{intersection layer} in \cite{MR2995793}. Let 
$\mathcal{I}$ denote the collection of all the intersection layers; at the
end of $n^{th}$ iteration, the collection $\mathcal{I}$ is updated as below: 
\begin{equation*}
\mathcal{I} := \{\mathcal{I}_{j,n}: j=1,\ldots,2^{n}\}.
\end{equation*}
The $(n+1)^{th}$ step makes use of $\mathcal{I}$ generated in the $n^{th}$
step to arrive at $\{\mathcal{I}_{j,n+1}:j=1,\ldots,2^{n+1}\}.$ Specific
details of how these random quantities are simulated can be found in \cite%
{MR2995793}. From the intersection layers $\mathcal{I}_{j,n}$ generated by
the algorithm at $n^{th}$ step, the dominating piecewise constant processes
can be formed as below: 
\begin{align*}
B^{\uparrow}_{n}(t) & = \sum_{j=1}^{2^{n}} U^{\uparrow}_{j,n}\mathbf{1}(t
\in((j-1)2^{-n}, j2^{-n}]), \text{ and } \\
B^{\downarrow}_{n}(t) & = \sum_{j=1}^{2^{n}} L^{\downarrow}_{j,n}\mathbf{1}%
(t \in((j-1)2^{-n}, j2^{-n}]).
\end{align*}
Further define the following piecewise linear process which shall serve as
our approximation for $B(\cdot):$ 
\begin{equation}
B_{n}(t) = \sum_{j=1}^{2^{n}}\left[ B((j-1)2^{-n}) +
2^{n}(B(j2^{-n})-B((j-1)2^{-n}))(t-(j-1)2^{-n}) \right] ,  \label{BM-APPROX}
\end{equation}
which is just a linear interpolation of the points
$\{B((j-1)2^{-n}):j=1,%
\ldots,2^{n}\}$
over the dyadic intervals in $[0,1].$ Note that all the random
variables used in the construction of
$B^{\uparrow}_{n}(\cdot), B^{\downarrow}_{n}(\cdot)$ and
$B_{n}(\cdot)$ are available in $\mathcal{I},$ and can be simulated on
a personal computer without any discretisation error. It is proved in \cite{MR2995793} that the dominating processes $%
B^{\uparrow }_{n}(\cdot)$ and $B^{\downarrow}_{n}(\cdot)$ have the following
convergence behavior: 
\begin{align}
\varepsilon_{n} := \sup_{t \in[0,1]}|B^{\uparrow}_{n}(t) -
B^{\downarrow}_{n}(t)| & = \max_{1 \leq j
\leq2^{n}}|U^{\uparrow}_{j,n}-L^{\downarrow }_{j,n}| \searrow0, \text{ and }
\label{BM-ERR} \\
\E \left[ \int_{0}^{1}|B^{\uparrow}_{n}(t) - B^{\downarrow
  }_{n}(t)|{d} t\right] & = O(2^{-n/2}) , \quad \text{ as } n
                          \rightarrow \infty.\notag
\end{align}

To generate a piecewise linear approximation of the $d$-dimensional
Brownian motion
$\mathbf{B}(\cdot) = (B_{1}(\cdot),\ldots,B_{d}(\cdot)),$ we generate
approximating processes
$B_{n,i}(\cdot),B^{\uparrow}_{n,i}(\cdot), \text{ and }
B^{\downarrow}_{n,i}(\cdot)$
independently for each 1-dimensional Brownian motion $B_{i}(\cdot)$ as
explained above, and use
\begin{equation*}
\mathbf{B}_{n}(t) = (B_{n,1}(t),\ldots,B_{n,d}(t)), \quad t \in[0,1]
\end{equation*}
as piecewise linear approximation for $\mathbf{B}(\cdot).$ Similar to
the 1-dimensional case, the simulated information is stored in the
intersection layers
$\mathcal{I} = \{ \mathcal{I}_{j,n}^i: j = 1,\ldots,2^n, i = 1,\ldots,
d\};$
here, $\mathcal{I}_{j,n}^i$ simply denotes the intersection layer
simulated to generate the approximation to $i^{th}$ component $B_i(t)$
of the driving Brownian motion $\mathbf{B}(\cdot).$ As in the
1-dimensional case, we use $\varepsilon_n$ to denote the error in
approximation at the $n^{th}$ step:
\begin{align}
  \varepsilon_n := \sup \left\{ \left \vert B^\uparrow_{n,i}(t) -
  B^\downarrow_{n,i}(t)\right\vert : t \in [0,1], i = 1,\ldots,d \right\}. 
  \label{error-d-dim}
\end{align}

\subsubsection{Generating $\varepsilon-$strong approximation of RBM.}  
\label{Sec-Str-Approx-RBM}
Given a linear path $(\mathbf{x}(t): t \in [t_0,t_1))$ specified by
initial condition $\mathbf{x}(t_0) \in \mathbb{R}^d$ and
$\dot{\mathbf{x}}(t) = \mathbf{m} \in \mathbb{R}^d$ for
$t \in (t_0,t_1),$ we explain in this section how to identify the
reflected path $\mathbf{y}(\cdot) = S(\mathbf{x}(\cdot))$ that solves
the Skorokhod problem in (\ref{SKO-PROB}) with
$(\mathbf{x}(t): t \in [t_0,t_1))$ as the path of the driving (free)
process. Once we know how to solve (\ref{SKO-PROB}) for a linear path
in the interval $(t_0,t_1),$ it is straightforward to iteratively
solve for any piecewise linear path of the driving (free) process.

As the slope of each component of the driving path $\mathbf{x}(\cdot)$
is fixed in the interval $[t_0,t_{1}),$ the slope of $\mathbf{y}$ at
time $t \in (t_0,t_{1}),$ denoted by $\mathbf{\dot{y}}(t),$ is
obtained by,
\begin{align*}
  \mathbf{\dot{y}}(t) = \mathbf{m} + R\mathbf{z}, 
\end{align*}
where, as per conditions 1)$-$3) in Skorokhod problem (\ref{SKO-PROB}),
$\mathbf{y}(t)$ and $\mathbf{z} \in \mathbf{R}^d$ should satisfy,
\begin{align*}
  \mathbf{y}(t) \geq \mathbf{0}, \quad \mathbf{z} \geq \mathbf{0},
  \quad \text{ and } 
  \mathbf{y}(t) \mathbf{z} = \mathbf{0}. 
\end{align*}
As a component $z_i$ of $\mathbf{z} = (z_1, \ldots, z_d)$ is nonzero
only when the respective component $y_i(t)$ of $\mathbf{y}(t)$ is
zero, it is useful to dynamically keep track of which of the
components of $\mathbf{y}(t) = (y_1(t),\ldots,y_d(t))$ are zero. We
accomplish this algorithmically by letting $C := \{ i : y_i(t) = 0\}.$
In addition, we use $R_C$ to denote the submatrix of $R$ formed by
letting $R_C = [R_{i,j}]_{i,j \in C};$ similarly, let
$\mathbf{z}_{_C} = [z_i]_{i \in C}, \mathbf{\dot{y}}_{_C}(t) =
[\dot{y}_i(t)]_{i \in C}$
denote the vectors formed by entries restricted to indices from the
set $C.$ As $\mathbf{z}_{_C}$ is the minimal non-negative vector that
maintains $\mathbf{y}_{_C}(t) =0$ (see \cite{Kella_Whitt}), it is
obtained by solving a linear program as in Algorithm
\ref{Alg-Str-Approx-RBM} when any of the components of $\mathbf{y}$
hit zero. 

\begin{algorithm}[h]{\small \caption{Algorithm to solve Skorokhod
      problem (\ref{SKO-PROB}) for a linear path
      $(\mathbf{x}(t): t\in [t_0,t_1))$ with component wise constant
      slopes $\mathbf{m}$ in the interval $[t_0,t_1)$ for given
      initial condition $\mathbf{y}(t_0)$} \begin{algorithmic}
      \\
      \Procedure{ApproxRBM}{$t_0,t_1,\mathbf{y}(t_0), \mathbf{m}$} \State
      Initialize the time data structure $\vec{t} = [t_0]$ and path
      space data structure $\vec{\mathbf{y}} = [\mathbf{y}(t_0)]$
      \State Initialize
      $\mathbf{y} = \mathbf{y}(t_0), C = \{i: y_i = 0\}, t' = t_0,
      \mathbf{z} = 0.$
      \While{$t' < t_1$} \State Solve the linear program
      $\min_{\mathbf{z}_{_C},
        \mathbf{\dot{y}}_{_C}} \mathbf{1}^T\mathbf{z},$\\
      \hspace{40pt} subject to
      $\mathbf{z}_{_C} \geq \mathbf{0}, \mathbf{\dot{y}}_{_C} =
      \mathbf{m}_{_C} + R_C\mathbf{z}_{_C}, \mathbf{\dot{y}}_C \geq
      \mathbf{0}.$
      \State Update \vspace{-16pt}
      \begin{align*}
        \mathbf{\dot{y}} &= \mathbf{m} + R\mathbf{z},\\
        t' &= \min\left\{t_1,  \min_{i: \dot{y}_i < 0}
             \left(t_0 - \frac{y_i}{\dot{y}_i} \right)\right\}, \hspace{110pt}\\
        \mathbf{y} &= \mathbf{y} + \mathbf{\dot{y}}(t' - t_0) \text{
                     and }\\
        C &= \{ i: y_i = 0\}.
      \end{align*}
      \State Append the time and path space data structures with the
      latest entries as in 
      \begin{align*}
        \vec{t}
     \leftarrow [\vec{t} \quad t'] \quad \text{ and } \quad  \vec{\mathbf{y}} \leftarrow
      [\vec{\mathbf{y}} \quad \mathbf{y}].
      \end{align*}
\EndWhile
\vspace{-20pt}
\State \textbf{end while}
\State \textbf{Return}  $\vec{t}, \vec{\mathbf{y}}.$ 
\EndProcedure
\end{algorithmic}
\label{Alg-Str-Approx-RBM} }
\end{algorithm}

Given component-wise constant slopes $\mathbf{m}$ of the driving
process $\mathbf{x}(\cdot)$ in the interval $(t_0, t_{1})$ and the
initial condition $\mathbf{y}(t_0) = \mathbf{y},$ Algorithm
\ref{Alg-Str-Approx-RBM} returns a vector $\vec{t}$ and a matrix
$\vec{y}$ of dimension $d \times \text{size}(\vec{t}),$ where
$\text{size}(\vec{t})$ is the dimension of $\vec{t}.$ The data
structures $\vec{t}$ and $\vec{\mathbf{y}}$ returned by Algorithm
\ref{Alg-Str-Approx-RBM} can be used to construct the piecewise linear
path $(\mathbf{y}(t): t \in [t_0,t_1))$ as follows: if the $i^{th}$
entry of $\vec{t}$ is $t_i$ and the $i^{th}$ column of the matrix
$\vec{y}$ is $\mathbf{y}_i,$ then linear interpolation of the points
$(t_i, \mathbf{y}_i)$ yields us the reflected path
$(\mathbf{y}(t): t \in [t_0,t_1)).$

\subsubsection{Detecting the interval $(\Tl, \Tr].$} Recall that the
objective of the preconditioning step is to simulate just enough
information about the RBM so that we arrive at the representation
(\ref{EQ_L_ZERO}). As explained in Section \ref{Sec-RAR-RBM}, we
propose to achieve this by detecting an interval
$(T_{{left}},T_{{right}}]$ containing $T$ such that the RBM stays in
the interior of positive orthant without hitting the reflecting
boundary anywhere in that interval. 
In order to accomplish this, we first make the following
observations:
\begin{enumerate}
\item[A)] 
  It is well-known that the Skorokhod map $S$ is lipschitz continuous
  (with respect to the uniform metric on the path space
  $C([0,1];\mathbb{R}^d)$) with Lipschitz constant
  $K_1 := (1-\alpha)^{-1};$ here $\alpha \in (0,1)$ denotes the
  spectral radius of the matrix $Q = I-R.$ (see \cite{MR606992}).
  Therefore, if we solve \eqref{SKO-PROB} with
  $\mathbf{X}(\cdot )=\mathbf{B}_{n}(\cdot )$ as the driving (free)
  process, the corresponding reflected process
  $\mathbf{Y}_n(\cdot) := S(\mathbf{B}_{n})(\cdot)$ satisfies,
\begin{equation}
  \Vert \mathbf{Y}_n - \mathbf{Y} \Vert :=  \sup_{t\in  [0,1], \  i = 1,\ldots,d}|Y_{n,i}(t) - Y_i(t)| <
  K_1\varepsilon _{n},
  \label{Lipsch-Conseq}
\end{equation}
where
$\varepsilon_n := \sup
\{B^{\uparrow}_{n,i}(t)-B^{\downarrow}_{n,i}(t):t \in [0,1], i =
1,\ldots,d\},$
and $B^{\uparrow}_{n,i}(\cdot), B^{\downarrow}_{n,i}(\cdot)$ are,
respectively, the simulated upper and lower bounding processes of the
$i^{th}$ component of the driving Brownian motion
$\mathbf{B}(\cdot).$ Since
\begin{align*}
  \Vert \mathbf{B}_n(t) - \mathbf{B}(t) \Vert =
  \sup\{\vert B_{n,i}(t) - B_i(t) \vert: t \in [0,1], i = 1,\ldots,d\}
  \leq \varepsilon_n,
\end{align*}
(\ref{Lipsch-Conseq}) follows as a simple consequence of Lipschitz
continuity of Skorokhod map. 

\item[B)] Let us say $a:= \min_{i=1,\ldots,d}Y_i(T) > 0;$ that is, the
  RBM $\mathbf{Y}(\cdot)$ lies in the interior of positive orthant at
  time $T.$ Then, due to continuity of Brownian motion paths, there
  exists a random interval $(T - t', T + t')$ where
  $\vert B_i(t) - B_i(T) \vert < a$ for all $t \in (T - t', T + t'),$
  and consequently,
  \begin{align*}
    &\mathbf{Y}(t) -  \mathbf{Y}(T)  = \mathbf{B}(t) - \mathbf{B}(T)
    \text{ and }\\
    &\mathbf{Y}(t) = \mathbf{Y}(T) + (\mathbf{B}(t) - \mathbf{B}(T)) >
  a\mathbf{1} -a\mathbf{1} =\mathbf{0},
  \end{align*}
  for all $t \in (T - t', T + t').$ In other words, the RBM
  $\mathbf{Y}(t)$ lies in the interior of positive orthant for every
  $ t \in (T-t', T+t')$ where $\vert B_i(t) - B_i(T) \vert < a.$ In
  order to make use of this observation, recall that our constructed
  piecewise constant upper and lower bounding processes
  $B_{n,i}^\uparrow$ and $B_{n,i}^{\downarrow}$ satisfy
  $B_{n,i}^{\uparrow }(\cdot)-B_{n,i}^{\downarrow }(\cdot)
  <\varepsilon_{n}$ for each component $i,$ and consequently,
\begin{equation*}
  |B_i(t)-B_i(s)|<\varepsilon _{n},\quad t,s\in \left[
    (j-1)2^{-n},j2^{-n}, 
  \right],  
\end{equation*}
for all $i = 1,\ldots,d,$ and $j = 1,\ldots,2^n.$ As a result, if
$\min_{i = 1,\ldots,d} Y_i(T)>\varepsilon _{n},$ then as
$|B_i(t) - B_i(T)| < \varepsilon_n$ for every $t$ in the dyadic
interval containing $T,$ the RBM $\mathbf{Y}(\cdot )$ does not hit the
reflecting boundary anywhere in the specific dyadic interval
$\left[(j-1)2^{-n}, \ j2^{-n}\right]$ containing
$T.$
\end{enumerate}

If we know that $\varepsilon_n$ is small enough such that every
component of the RBM satisfies $Y_i(T) \geq \varepsilon _{n},$ then
from Observation B) noted above, the RBM $\mathbf{Y}(t)$ stays in the
interior of positive orthant for every $t$ in the dyadic interval
$((j-1)2^{-n},j2^{-n}]$ containing $T;$ consecutively, we can declare
the corresponding interval $((j-1)2^{-n},j2^{-n}]$ as $(\Tl, \Tr].$
However, since we do not know $\mathbf{Y}(T),$ the immediate objective
is to figure out how to guarantee that $Y_i(T)$ is indeed larger than
$\varepsilon _{n},$ for every $i = 1,\ldots,d.$ From the Lipschitz
continuity in observation A), if $\varepsilon_n$ is small enough so
that $Y_{n,i}(T) = S(\mathbf{B}_{n})(T)>(K_1+1)\varepsilon _{n}$ for
some $n,$ then $Y_i(T)>\varepsilon _{n}.$ Since $\mathbf{Y}(t)$ lies
in the interior of the positive orthant almost everywhere, we will
indeed have that $Y_{n,i}(T)>(K_1+1)\varepsilon _{n}$ for suitably
small approximation error $\varepsilon_n.$ Now define,
\begin{equation*}
  N:=\inf \{n\geq 1: \ Y_{n,i}(T) > (K_1+1)\varepsilon _{n}, i = 1,\ldots,d\},\text{ and }\delta
  :=\varepsilon_{_N}.
\end{equation*}
Recall that
$\varepsilon _{n}:=\sup\{|B_{n,i}^{\uparrow }(t)-B_{n,i}^{\downarrow
}(t)| : t \in [0,1], i = 1,\ldots, d\}.$
The preconditioning procedure for detecting the interval
$(T_{{left} },T_{{right}}]$ that simulates approximations to Brownian
motion and RBM until the stopping time $N$ is summarized in Algorithm 
\ref{APPROX-RBM}.

\begin{algorithm}[h]{\small \caption{Preconditioning step to arrive at
      the representation (\ref{EQ_L_ZERO}).\newline (achieved mainly
      via $\varepsilon-$strong approximation of the underlying
      stochastic processes)} \begin{algorithmic}
      \\
      \Procedure{Preconditioning}{$T$} 
      \State Initialize $\mathcal{I} = \emptyset, n=0, \delta = 1,
      \mathbf{Y}_0(t) = 0, t \in [0,1]$
      \While{$n = 0$ \ \  {\tt OR} \ \  ${Y}_{n,i}(T) < (K_1+1) \delta$ for some $i \in \{1,\ldots,d\}$} 
      \State Increment $n \longleftarrow n+1$ 
      \State Simulate the intersection layers
      $\Ijn^i$ for $j=1,\ldots,2^n, i =1,\ldots,d$ and the piecewise
      constant upper and lower bounding 
      processes $B_{n,i}^\uparrow(\cdot), B_{n,i}^\downarrow(\cdot),$
      conditional on $\mathcal{I}$  
     \State Form $\mathbf{B}_n(\cdot) = (B_{n,1},\ldots,B_{n,d})$ as
     in \eqref{BM-APPROX}, which serves as a  
      piecewise linear approximation to $\mathbf{B}(\cdot)$ 
     \State Update \vspace{-16pt}
     \begin{align*}
       \delta &= \sup\{ \vert B_{n,i}^\uparrow(t) -
                B_{n,i}^\downarrow(t)\vert: t \in [0,1], i =
                1,\ldots,d\} \hspace{35pt}\\ 
       \mathcal{I} &= \{\Ijn:j=1,\ldots,2^n\}
     \end{align*}
    \State Letting  $\mathbf{X}(\cdot)=\mathbf{B}_n(\cdot)$ in \eqref{SKO-PROB}, solve for the
      reflected process $\mathbf{Y}(\cdot)$ using Algorithm
      \ref{Alg-Str-Approx-RBM}; call the solution as
      $\mathbf{Y}_n(\cdot)$ 
\EndWhile
\vspace{-5pt}
\State \textbf{end while}
\State Find $J = \{1 \leq j \leq 2^n: (j-1)2^{-n} < T \leq
j2^{-n}\}$
\State Set $N = n, \Tl = (J-1)2^{-n}$ and $\Tr = J2^{-n}$
\State \textbf{Return} $\mathcal{I}, N, \Tl, \Tr, \delta$ and
$\mathbf{Y}^\delta(\cdot) := \mathbf{Y}_n(\cdot)$
\EndProcedure
\end{algorithmic}
\label{APPROX-RBM} }
\end{algorithm}


With this construction, since the RBM $\mathbf{Y}(\cdot)$ does not hit
the reflecting boundary anywhere in the interval $(T_{{left}},T_{{%
    right}}],$
the dynamics of $\mathbf{B}(\cdot)$ and $\mathbf{Y}(\cdot)$ match in
$(T_{{left}},T_{{right}}];$ in particular, 
\begin{equation}
\mathbf{Y}(t) - \mathbf{Y}(T_{{left}}) = \mathbf{B}(t) - \mathbf{B}(T_{{left}%
}), \text{ for all } t \in(T_{{left}},T_{{right}}],
\label{BM-RBM-MATCHING-DYNAMICS}
\end{equation}
thus resulting in the desired additive
representation 
$\mathbf{Y}(T) = \mathbf{Y}(T_{{left}}) + \mathbf{\Delta},$ 
in (\ref{EQ_L_ZERO}), where the increment
$\mathbf{\Delta} := \mathbf{B}(T)-\mathbf{B}(T_{{left}})$ is simply
the Brownian increment.

\subsection{The conditional probability density of $\Delta$ for the
  refine  until accept / reject sampler} 
\label{Sec-Cond-Dens-Delta}
The requirements R1') and R2'), listed in Section \ref{Sec-RAR-RBM}
for the implementation of the refine until accept / reject exact
sampler, necessitate us to know the law of $\mathbf{\Delta}$
conditional on all the simulated collection of random variables
$\mathcal{I}.$ For ease of exposition, we consider the 1-dimensional
case; the conditional probability density of the Brownian increment
$\mathbf{\Delta}$ in $d$-dimensions, denoted by
$f_{\mathbf{\Delta}}(\cdot),$ is given simply by the product form of
1-dimensional densities.

At any stage of algorithm, all the information simulated about the
driving Brownian motion are available in the intersection layers
$\mathcal{I} = \{\mathcal{I}_{j,n}: j = 1, \ldots, 2^n\}.$ From
Algorithm \ref{APPROX-RBM}, recall that $J$ is the index corresponding
to the dyadic interval $(T_{{left}},T_{{right}}]$ that contains $T;$
that is $(J-1)2^{-N}=T_{{%
    left}}<T\leq T_{{right}}=J2^{-N}.$ For ease of notation, let
\begin{align*}
  L^{\downarrow }:=L_{J,N}^{\downarrow }-B(T_{{left}}),\hspace{3pt}&
  L^{\uparrow }:=L_{J,N}^{\uparrow }-B(T_{{left}}), \\
  U^{\downarrow }:=U_{J,N}^{\downarrow }-B(T_{{left}}),\hspace{3pt}&
  U^{\uparrow }:=U_{J,N}^{\uparrow }-B(T_{{left}}), \\
  l:=T_{{right}}-T_{{left}},\hspace{3pt}s:=T-T_{{left}}& \text{ and
  }v:=B(T_{{%
      right}})-B(T_{{left}}).
\end{align*}%
Further let $W(\cdot )$ denote an independent standard Brownian motion on $%
C[0,1]$ under measure $\Pr \left( \cdot \right) .$ Then due to Markov
property of $B(\cdot ),$ the increment $\Delta $ conditional on $\mathcal{I}$
has the following density: 
\begin{equation}
f_{\Delta }(x)dx=\Pr \left\{ W(s)\in {d}x\left. \frac{{}}{{}}\right\vert
W(l)=v,\inf_{0\leq t\leq l}W(t)\in (L^{\downarrow },L^{\uparrow
}),\sup_{0\leq t\leq l}W(t)\in (U^{\downarrow },U^{\uparrow })\right\} .
\end{equation}%
Note that the support of $f_{\Delta }(\cdot )$ is
$(L^{\downarrow },U^{\uparrow }).$ A closed form expression for
$f_{\Delta }(\cdot )$ follows from Proposition 5.1 of
\cite{MR2995793}, and is given here:
\begin{equation*}
f_{\Delta }(x)\propto \rho (x)\times \pi (x),
\end{equation*}%
where for any fixed $L^{\downarrow },L^{\uparrow },U^{\downarrow
},U^{\uparrow },v,s\text{ and }l,$ 
\begin{align}
\pi (x)& :=\exp \left( -\frac{1}{2}\left( x-\frac{s}{l}v\right) ^{2}\left. 
\frac{{}}{{}}\right/ \left( \frac{s(l-s)}{l}\right) \right) ,\text{ and }
\label{PI-EXP} \\
\rho (x)& :=\Pr \left\{ \inf_{0\leq t\leq l}W(t)\in (L^{\downarrow
},L^{\uparrow }),\sup_{0\leq t\leq l}W(t)\in (U^{\downarrow },U^{\uparrow
})\left. \frac{{}}{{}}\right\vert W(s)=x,W(l)=v\right\}   \notag \\
& =\gamma _{1}(x)\gamma _{2}(x)-\gamma _{3}(x)\gamma _{4}(x)-\gamma
_{5}(x)\gamma _{6}(x)+\gamma _{7}(x)\gamma _{8}(x).  \label{RHO-GAMMA-EXP}
\end{align}%
To define $\gamma _{1},\ldots ,\gamma _{8},$ first consider the
probability that the Brownian bridge from $a$ to $b$ in the time
interval $%
[0,r]$ stays within $(L,U):$
\begin{align}
\gamma (L,U;r,a,b)& :=\Pr \left\{ L<\inf_{0\leq t\leq l}W(t)\leq \sup_{0\leq
t\leq l}W(t)<U\left. \frac{{}}{{}}\right\vert W(0)=a,W(r)=b\right\}   \notag
\\
& =\left( 1-\sum_{j=1}^{\infty }(\sigma _{j}-\tau _{j})\right) \mathbf{1}%
(a,b\in (L,U)),  \label{EXIT-PROB-EXPRESSION}
\end{align}%
\begin{align*}
\text{where, }\sigma _{j}& :=\exp \left( -\frac{2}{r}%
((U-L)j+L-a)((U-L)j+L-b)\right)  \\
& \hspace{50pt}+\exp \left( -\frac{2}{r}((U-L)j-U+a)((U-L)j-U+b)\right) ,%
\text{ and } \\
\tau _{j}& :=\exp \left( -\frac{2(U-L)j}{r}((U-L)j+a-b)\right) +\exp \left( -%
\frac{2(U-L)j}{r}((U-L)j+b-a)\right) .
\end{align*}%
The expression \eqref{EXIT-PROB-EXPRESSION} for $\gamma (L,U;l,a,b)$ is
originally from \cite{MR1816120}. Now we are ready to define $\gamma
_{1},\ldots ,\gamma _{8}$ mentioned in \eqref{RHO-GAMMA-EXP}: 
\begin{align*}
\gamma _{1}(x)& =\gamma (L^{\downarrow },U^{\uparrow };s,0,x),\hspace{17pt}%
\gamma _{2}(x)=\gamma (L^{\downarrow },U^{\uparrow };l-s,x,v),\hspace{15pt}%
\gamma _{3}(x)=\gamma (L^{\uparrow },U^{\uparrow };s,0,x), \\
\gamma _{4}(x)& =\gamma (L^{\uparrow },U^{\uparrow };l-s,x,v),\hspace{1pt}%
\gamma _{5}(x)=\gamma (L^{\downarrow },U^{\downarrow };s,0,x),\hspace{31pt}%
\gamma _{6}(x)=\gamma (L^{\downarrow },U^{\downarrow };l-s,x,v), \\
\gamma _{7}(x)& =\gamma (L^{\uparrow },U^{\downarrow };s,0,x),\hspace{17pt}%
\gamma _{8}(x)=\gamma (L^{\uparrow },U^{\downarrow };l-s,x,v). 
\end{align*}
To perform acceptance / rejection type-sampling, we need that the
conditional density $f_{\Delta}$ (of the Brownian increment) is
Lipschitz continuous (see Requirement R1') in Section
\ref{Sec-RAR-RBM}. Lemma \ref{LIPSH-CONT} is a step towards
establishing this fact.
\begin{lemma}
There exists positive constants $c_{\pi},K_{\pi},c_{\rho}$ and $K_{\rho}$
such that for any fixed $L^{\downarrow},L^{\uparrow},U^{\downarrow},U^{%
\uparrow},$ $v,s$ and $l,$ 
\begin{align*}
\pi(x) < c_{\pi} & ,\hspace{10pt} |\pi(x)-\pi(y)| < K_{\pi}|x-y| \text{ and }
\\
\rho(x) < c_{\rho} & ,\hspace{10pt} |\rho(x)-\rho(y)| < K_{\rho}|x-y|,
\end{align*}
for all $x,y \in(L^{\downarrow},U^{\uparrow}).$ 
\label{LIPSH-CONT}
\end{lemma}
Explicit closed-form expressions for the constants
$K_{\pi },K_{\rho },c_{\pi }$ and $c_{\rho }$ are presented in the
Appendix. 

Following the representation that $Y(T) = Y(\Tl) + \Delta,$ the
conditional density of $Y(T)$ given $Y(\Tl)$ is given by
$f_{\Delta}(w-Y(\Tl)),$ which is supported on
$(Y(\Tl) + L^\downarrow, Y(\Tl) + U^\uparrow) \subset (Y(\Tl)- \delta,
Y(\Tl) + \delta).$ 
As the unknown $Y(\Tl)$ differs from $Y^\delta(\Tl)$ only by
$\pm K_1 \delta,$ the support of $Y(T)$ conditional on the simulated
information is, in turn, a subset of
\[ \big(Y^{\delta}(\Tl) - (K_1 + 1) \delta, \ Y^{\delta}(\Tl) + (K_1 +
1)\delta \big). \]
Here, recall that $K_1 := (1-\alpha)^{-1}$ is the Lipschitz constant
of the Skorokhod reflection map $S.$ Consecutively, if we propose a
sample $Z$ from the uniform distribution in the interval
$(Y^{\delta}(\Tl) - (K_1 + 1)\delta, \ Y^{\delta}(\Tl) + (K_1 +
1)\delta),$
then the likelihood ratio (or the Radon-Nikodym derivative) between
the true conditional density and proposal density is proportional to
$f_{\Delta}(w-Y(\Tl));$ consecutively, a traditional accept / reject
algorithm would accept the proposed sample $Z$ if
\begin{align*}
  V < L\big(Z; Y(\Tl)\big) := \frac{\pi \big(Z -Y(\Tl) \big) \rho \big(Z -
  Y(\Tl)\big)}{c_\pi c_\rho},  
\end{align*}
for an independent $V \sim \text{Unif}(0,1);$ hereafter, we use
$L(z;y)$ to denote
\begin{align*}
  L(z;y) := \frac{\pi(z-y) \rho(z-y)}{c_\rho c_\phi}.
\end{align*}
However, as we do not know $Y(\Tl)$ exactly, if $L(Z;\ \cdot)$ is
Lipschitz continuous, the fact that $Y(\Tl)$ differs from
$Y^\delta(\Tl)$ by $\pm K_1\delta,$ can be used to implement a refine
until accept / reject sampler discussed in
Section~\ref{Sec-RAR-Intro}.

The Lipschitz continuity of $L(z;y)$ as a function of $y,$ follows as
a simple consequence of Lemma \ref{LIPSH-CONT} established earlier. If
$f$ and $g$ are Lipschitz continuous with Lipschitz constants $K_f$
and $K_g,$ and respective absolute bounds $c_f$ and $c_g,$ then $fg$
is a Lipschitz continuous function with Lipschitz constant at most
$c_fK_g + c_gK_f.$ As a result, the product
$\pi(\cdot) \times \rho(\cdot)$ is Lipschitz continuous as well with
Lipschitz constant $c_\pi K_\rho + c_\rho K_\pi.$ Consequently, the
function $L(z,y)$ is Lipschitz continous, as a function of $y,$ with
Lipschitz constant
\begin{align*}
  K_2 := \frac{K_\pi}{c_\pi} + \frac{K_\rho}{c_\rho}. 
\end{align*}


Given this Lipschitz continuity of $L,$ one can unambiguously accept
the proposal $Z$ if,
\begin{align*}
  V <  L\left(Z;Y^\delta(\Tl)\right)- K_1 K_2\delta,  
\end{align*}
or reject the proposal $Z$ conclusively if
\begin{align*}
    V > L\left(Z;Y^\delta(\Tl)\right)+ K_1 K_2\delta.  
\end{align*}
However, if $V$ is within
$L\left(Z;Y^\delta(\Tl)\right) \pm K_1 K_2 \delta,$ then we obtain a
more refined approximation $Y^{\varepsilon_n}(\Tl),$ for a suitable
$\varepsilon_n < \delta,$ that is good enough to decide whether to
accept / reject. In particular, if $\varepsilon_n$ is smaller than
$(K_1K_2)^{-1} \left\vert V - L\left(Z;Y^{\varepsilon_n}(\Tl)\right)
\right\vert,$
then 
\begin{align}
  V < L\big(Z;Y(\Tl)\big) \text{ if and only if } V <
  L(Z;Y^{\varepsilon_n}(\Tl))- K_1 K_2\varepsilon_n.  
\label{Equiv-comp}
\end{align}
This equivalent comparison is at the heart of the refine until accept
/ reject sampler in Algorithm \ref{REJ-SAMP-ALGO} below. It takes the
intersection layers $\mathcal{I}$ returned by Algorithm
\ref{APPROX-RBM} as input, and generates further refined
approximations $(Y^{\varepsilon_n}(t):t\in \lbrack 0,T_{{left}}])$ of
the RBM, if necessary, in order to perform the equivalent comparisons
in (\ref{Equiv-comp}). 

\begin{algorithm}[h] {\small \caption{ To accept/reject the proposal
      $Z \sim \text{Unif}(Y^\delta(\Tl) - (K_1 + 1)\delta,
      Y^\delta(\Tl) + (K_1 + 1)\delta).$
      If $V < L\big(Z;Y(\Tl)\big)$, the algorithm returns $Z$;
      otherwise it rejects Z and returns nothing. Recall that
      $L(z;y) := {c_\pi^{-1} c_\rho^{-1}\pi(z-y)\rho(z-y)},$
      $K_1 := 1/(1-\alpha)$ and
      $K_2 := {K_\pi}/{c_\pi} + K_\rho/c_\rho.$}
\begin{algorithmic}
\\
\Procedure{RAR-Sampler}{$\mathcal{I},\Tl,N, \delta, Y^\delta(\Tl),Z$}
\State Initialize $n = N,\ \varepsilon_n = \delta$
\State Draw $V$ uniformly from $[0,1]$
\While{$\varepsilon_n > (K_1K_2)^{-1}\vert V - L\big(Z;Y^{\varepsilon_n}(\Tl)\big)\vert$}
\State Increment $n \longleftarrow n+1$
\State Simulate the intersection layers $\Ijn$ for $j=1,\ldots,2^n$ conditional
on $\mathcal{I}$
\State Form $B_n(\cdot)$ as in \eqref{BM-APPROX}, which serves as
piecewise linear approximation to $B(\cdot)$
\State Set $\varepsilon_n = \max\{ \Uup_{j,n}-\Ldown_{j,n}:
j=1,\ldots,2^n\}$ and  $\mathcal{I} = \{\Ijn:j=1,\ldots,2^n\}$
\State Letting $X(t)=B_n(t)$ for $t \in [0,\Tl]$ in \eqref{SKO-PROB}, solve for
the reflected process $Y(\cdot);$ call the solution as $Y_n(\cdot);$
set the required refined approximation $Y^{\varepsilon_n}(\Tl) = Y_n(\Tl).$  
\EndWhile
\State \textbf{end while}
\If {$V < L\big(Z; Y^{\varepsilon_n}(\Tl)\big) - K_1 K_2\varepsilon_n$} 
\State \textbf{Return} $Z$
\Else
\State \textbf{Return} $\emptyset$
\EndIf
\State \textbf{end if}
\EndProcedure
\end{algorithmic}
\label{REJ-SAMP-ALGO} }
\end{algorithm}

For $d$-dimensional processes, the probability density of the
increment $%
\mathbf{\Delta}$
and proposal density $g(\cdot)$ are both given by product of
1-dimensional densities. This results in a likelihood ratio which is
also of product form, leading to a straightforward generalization of
the refine until accept / reject procedure given in Algorithm
\ref{REJ-SAMP-ALGO}.

\section{A note on computational complexity}
\label{Sec-Complexity}
Our objective in this section is to understand the computational
effort required to execute the refine until accept / reject exact
sampler described in Section \ref{SEC-SAMP-PROC}. For ease of
exposition, we do not keep track of multiplying constants, and instead
adopt the following standard notation to describe the asymptotic
behaviour of functions: For given functions
$f:\mathbb{R}^+ \rightarrow \mathbb{R}^+$ and
$g:\mathbb{R}^+ \rightarrow \mathbb{R}^+,$ we say $f(x) = O(g(x))$ if
there exists $c_1 > 0$ and $x_1$ large enough such that
$f(x) \leq c_1g(x)$ for all $x > x_1$; further, we say
$f(x) = \Theta(g(x))$ if there also exists $c_2 > 0$ and $x_2$ large
enough such that $c_2g(x) \leq f(x) \leq c_1g(x)$ for all $x >
x_2.$


Recall the definition of error in approximation $\varepsilon_n$ in
(\ref{error-d-dim}), and to achieve this accuracy we needed to
simulate relevant information (such as maxima, minima and endpoints
for $d$ independent 1-dimensional Brownian motions
$\mathbf{B}_i(\cdot)$) in $2^n$ dyadic intervals as described in
Section \ref{SEC-STR-APPROX}. As this entails constant amount of
expected computation for each dimension in each of the $2^n$ dyadic
intervals, the computational cost at the end of $n$ steps of the
iterative procedure described in Section \ref{SEC-STR-APPROX} is
$\Theta(d2^n).$ 

Next, observe that the preconditioning step requires us to iterate
until the stopping time
$N = \inf \{ n \geq 1: Y_{n,i}(t) > (K_1 +1)\varepsilon_n, i =
1,\ldots,d \}.$
As the computational cost of solving the Skorokhod problem (as in
Algorithm \ref{Alg-Str-Approx-RBM}) with a piecewise linear input
$\mathbf{B}_k(\cdot)$ is $\Theta(d^32^k)$ uniformly in $d,$ the total
cost of executing the entire preconditioning step is
$\Theta( d^3\sum_{k = 1}^N 2^k) = \Theta (d^32^N).$ Here, we have used
that $O(d^3)$ computations are needed to solve the linear program in
Algorithm \ref{Alg-Str-Approx-RBM} when the set $C$ contains $O(d)$
elements. Following the same line of reasoning, if we let
\begin{align*}
N' = \inf\{ n \geq N: \varepsilon_n < (K_1K_2)^{-1} \vert V -
  L(\mathbf{Z};\mathbf{Y}^{\varepsilon_n}(\Tl))\vert \},  
\end{align*}
as required in the refine until accept / reject step in Algorithm
\ref{REJ-SAMP-ALGO}, the corresponding computational cost is
$\Theta(d^32^{N'}).$ As $N' \geq N$, the total computational cost of
the sampling procedure is
$\Theta(d^32^N + d^32^{N'}) = \Theta(d^32^{N'}).$ In other words,
there exists positive constants $c_1$ and $c_2$ such that
\begin{align*}
  c_1d^32^{N'} \leq \text{ computational cost } \leq c_2d^32^{N'}. 
\end{align*}
Therefore, the expected computational cost of the entire sampling
procedure is $\Theta(d^3E[2^{N'}]).$

Next, to compute $E[2^{N'}],$ we first use the definition of
$\varepsilon_n$ in (\ref{error-d-dim}) to observe that
\begin{align*}
  P \left( N' > n \right) = P\left( \varepsilon_n >
  \frac{D}{K_1K_2}\right)
  = P\left( \max_{{i=1,\ldots,d}}\sup_{t \in [0,1]} \vert B_{n,i}^\uparrow(t)
    - B_{n,i}^{\downarrow}(t) \vert > \frac{D}{K_1K_2}\right), 
\end{align*}
where we have let
$D = \vert V - L(\mathbf{Z};\mathbf{Y}^{\varepsilon_n}(\Tl)) \vert$
for notational convenience. For each fixed $i,$ it follows from the
construction of 1-dimensional piecewise constant bounding processes
$B_{n,i}^\uparrow(\cdot)$ and $B_{n,i}^\downarrow$ in Section
\ref{SEC-STR-APPROX} that
\begin{align*}
  \sup_{t \in [0,1]}\vert B_{n,i}^\uparrow(t) - B_{n,i}^{\downarrow}(t) \vert &= \max_{j
  = 1,\ldots,2^n} \vert U_{j,n}^\uparrow - L_{j,n}^\downarrow \vert\\
  &\leq \max_{j = 1,\ldots,2^n} \left\{ \vert U_{j,n}^\uparrow - M_{j,n} \vert + \vert M_{j,n} -
    m_{j,n} \vert + \vert m_{j,n} - L_{j,n}^\downarrow \vert \right\}
  \\
  &\leq \max_{j = 1,\ldots,2^n} \left\{ 2^{-(n+1)/2} + \vert M_{j,n} -
    m_{j,n} \vert + 2^{-(n+1)/2} \right\}\\
  &\overset{D}{=} 2^{-n/2} + \max_{j=1,\ldots,2^n} \left(\sup_{t \in
    [0,2^{n}]} W_j(t) - \inf_{t \in [0,2^{n}]} W_j(t) \right)
\end{align*}
where $(W_j(t):j = 1,\ldots,2^n)$ are $2^n$ independent copies of
standard Brownian motion.  Here, the notation $\overset{D}{=}$ is used
to denote equality in distribution. If we let
\[\bar{Z}_j := 1 + \sup_{t \in [0,1]} W_j(t) - \inf_{t \in [0,1]}
W_j(t),\] then due to self-similarity of Brownian motion,
\begin{align*}
  \sup_{t \in [0,1]}\vert B_{n,i}^\uparrow(t) - B_{n,i}^{\downarrow}(t) \vert 
  &\overset{D}{=} 2^{-n/2}  + 2^{-n/2} \left(\sup_{t \in
    [0,1]} W_j(t) - \inf_{t \in [0,1]} W_j(t)\right) \\
  &=  2^{-n/2}  \max_{j = 1,\ldots,2^n} \bar{Z}_j.
\end{align*}
Since the approximations $B_{n,i}$ are independently obtained for each
$i = 1,\ldots,d,$ we have 
\begin{align*}
  \sup_{t \in [0,1]}\vert B_{n,i}^\uparrow(t) -
  B_{n,i}^{\downarrow}(t) \vert = \max_{k = 1,\ldots,d2^n} \bar{Z}_k, 
\end{align*}
where $(\bar{Z}_k: k =1,\ldots,d2^n)$ are $d2^n$ independent copies of
$\bar{Z} := 1+\sup_{t \in [0,1]}W(t) - \inf_{t \in [0,1]}W(t).$ For
notational convenience, let us denote
$M_n = \max_{j = 1,\ldots,d2^n} \bar{Z}_j.$ Then
\begin{align*}
  P \left( N' > n\right) &= P \left( 2^{-n/2} M_n> \frac{D}{K_1K_2} \right)
  = P \left( D < K_1K_2 \frac{M_n}{2^{n/2}} \right)\\
  &= P\left( L(\mathbf{Z};\mathbf{Y}^{\varepsilon_n}(\Tl)) - K_1K_2
    \frac{M_n}{2^{n/2}} < V <
    L(\mathbf{Z};\mathbf{Y}^{\varepsilon_n}(\Tl)) + K_1K_2
    \frac{M_n}{2^{n/2}}\right),
\end{align*}
because of our earlier definition that $D = \vert V -
L(\mathbf{Z};\mathbf{Y}^{\varepsilon_n}(\Tl)) \vert.$ As $V \sim
\text{Unif}[0,1],$ it is immediate that 
\begin{align*}
  P \left( N' > n\right) &= \frac{2K_1K_2}{2^{n/2}}E \left[ M_n
                           \right] = \Theta\left( \sqrt{\frac{n}{2^n}}
  \right), 
\end{align*}
where the second equality follows from the observation that
$E[M_n] = \Theta( \sqrt{n}),$ which is proved in Lemma
\ref{Lem-exp-max-BM} in appendix.  Therefore,
$P(2^{N'} > x) = \Theta(x^{-0.5}\sqrt{\log x}).$ As the random
variable $2^{N'}$ has regularly varying tails with index $-0.5,$ 
\begin{align*}
  E \left[ 2^{N'} \right] = \infty,
\end{align*}
and consequently, expected total computational cost is infinite.

An alternative, intuitive explanation for why the expected termination
time is infinite is as follows: Note that conditional on $\mathbf{Z}$
and $\mathbf{Y}^{\varepsilon_n}(\Tl)$, the distance
$D=\left\vert V-L(\mathbf{Z}; \mathbf{Y}^{\varepsilon_n}(\Tl)
\right\vert $
is less than $\delta $ with probability $O\left( \delta \right)$
(because $V$ is uniformly distributed). Thus, if the cost of
generating $\mathbf{Y}^{D}$ (required to decide whether to accept or
reject) is $C\left( D\right) $, the running time of the algorithm
would be finite if $\int_{0}^{1}C\left( u\right) du < \infty .$
Unfortunately, however, the cost of producing an $\varepsilon $-strong
approximation to Brownian motion ($\mathbf{X}^{\varepsilon }$) is
roughly $O\left( 1/\varepsilon ^{2}\right) $ (see, for example,
\cite{MR2995793}) and therefore $C\left( D\right) \geq c/D^{2},$ with
positive probability, for some $c>0$, which results in an infinite
expected running time.

\section{Conclusions} 
We provide the first exact sampling algorithm to obtain samples from a
multi-dimensional reflected Brownian motion. The algorithm relies on a
novel conditional acceptance / rejection step, which is
implemented 
by carefully refining $\varepsilon-$strong approximations of the
reflected Brownian motion path $\mathbf{Y}(\cdot)$ until we can
conclusively accept or reject a proposal $\mathbf{Z}$ from a suitable
uniform distribution. Unfortunately, as shown in Section
\ref{Sec-Complexity}, the proposed algorithm has expected termination
time because of the large amount of computational effort required to
conclusively decide whether $V < L(\mathbf{Z}; \mathbf{Y}(\Tl))$ when
the proposal likelihood $L(\mathbf{Z}; \mathbf{Y}(\Tl))$ and the
uniform random variable $V$ are close. It may be of interest to the
readers to know whether the entire exact sampling scheme can be
executed with finite expected computational effort if we, somehow, are
able to resolve the difficulty in deciding whether
$V < L(\mathbf{Z}; \mathbf{Y}(\Tl))$ with finite expected
computational effort. We believe this is indeed the case because of
the following reasoning.

Apart from the refine until accept / reject step in Algorithm
\ref{REJ-SAMP-ALGO}, the only other step where we execute a `while'
loop performing a random comparison is in the pre-conditioning
procedure in Algorithm \ref{APPROX-RBM}. 
Recall that the preconditioning step must develop
a 
piecewise approximation to Brownian motion that is accurate enough to
satisfy
$\varepsilon_n := \Vert \mathbf{B}_n(\cdot) - \mathbf{B}(\cdot) \Vert
< (K_1 + 1)Y_i(T),$
for $i = 1,\ldots,d,$ in order to identify $(\Tl,
\Tr).$ 
If the probability density of $Y_i(T)$ evaluated at 0 is positive (as
in the 1-dimensional RBM case), then the probability that $Y_i(T)$ is
smaller than $\delta$ is at least $c_1\delta$ (for every
$\delta < \delta'$ suitably small), and the computational effort
required to generate a Brownian approximation that satisfies
$\varepsilon_n < (K_1 + 1)\delta$ is larger than $c_2/\delta^2$ with
positive probability; here, $c_1$ and $c_2$ are suitable positive
constants. As the required computational effort
$O(1/(\min_i Y_i(T))^2)$ is high when $Y_i(T)$ is close to $0$ for
some $i,$ the expected computational effort required in the
preconditioning step is at least
\[\int_{c_2/\delta'^2}^{\infty} P(\text{computational cost} > u)du \geq
\int_{c_2/\delta'^2}^\infty c_1\sqrt{c_2/u}\ du = \infty.\]
However, this difficulty can be easily resolved if we imagine, for a
moment, that it is possible to resolve the earlier difficulty
explained in Section \ref{Sec-Complexity} (on deciding whether
$V < L(\mathbf{Z};Y(\Tl))$ within finite expected time), and it is
possible to obtain exact samples of $\mathbf{Y}(T)$ whenever
$\mathbf{Y}(T)$ is bounded away from the reflecting barrier;
specifically, let us assume we can obtain samples of $\mathbf{Y}(T)$
with finite expected computational effort $C_\gamma$ when
$\min_i Y_i(T) > \gamma$ for some fixed constant $\gamma \in (0,1).$
In that case, we 
first obtain an exact sample of
$\mathbf{Y}(t')$ for the latest $t' \leq T$ during which 
$\min_i \mathbf{Y}_i(t') > \gamma.$ To be specific, define 
$t' := \sup\{t \leq T: \min_i Y_i(t) > \gamma\},$
$\tilde{\mathbf{Y}} := \mathbf{Y}_i(t')$ and $T_\gamma := T - t'.$  
Then due to the Markov property and self-similarity of RBM, the
original objective of obtaining a sample of $\mathbf{Y}(T)$ can be
equivalently written as follows: Obtain a sample of
$\mathbf{Y}_{new}(T_\gamma/\gamma^2),$ where $\mathbf{Y}_{new}(\cdot)$ is
also an RBM obtained by shifting and scaling the RBM
$\mathbf{Y}(\cdot)$ as in
$\mathbf{Y}_{new}(t) := \mathbf{Y}(t' + \gamma^2 t )/\gamma$ with
initial condition $\mathbf{Y}_{new}(0) := \tilde{\mathbf{Y}}/\gamma.$
Thus, even if the original problem of detecting $(\Tl, \Tr)$ is
difficult when $\mathbf{Y}(T) < \gamma,$ by a suitable translation and
scaling (magnification) of the underlying Brownian and RBM paths, we
have a new, but equivalent, objective of sampling from
$\mathbf{Y}_{new}(T_\gamma/\gamma^2) = \mathbf{Y}(T)/\gamma.$ In case if
$\mathbf{Y}_{new}(T_\gamma/\gamma^2)$ is smaller than $\gamma$ as well, we
perform a similar translation and scaling once again
recursively. 
Since $E[T_\gamma/\gamma^2]$ is uniformly bounded as a function of
$\gamma,$ a simple recursive algorithm complexity analysis (see, for
example, Chapter 4 of \cite{Cormen:2001:IA:580470}) yields us that the
total expected computational effort of the described recursive
procedure is $O(C_\gamma/P(\min_i Y_i(T) > \gamma)),$ which is finite,
as per our assumption on $C_\gamma$.

As the described recursive construction for the preconditioning step
is built on the assumption that we can conclusively decide whether
$V < L(\mathbf{Z}; \mathbf{Y}(\Tl))$ within finite expected
computational effort, we identify the difficulty explained carefully
in Section \ref{Sec-Complexity} as the only fundamental bottleneck in
obtaining exact samples of multi-dimensional RBM. Future research that
addresses this bottleneck by means of new techniques will make the
proposed algorithm, which is currently of theoretical importance, to
be more suitable for practice as well.

\appendix

\section*{Appendix}

Here we provide the proof of Lemma \ref{LIPSH-CONT}, and present explicit
expressions for the constants $K_{\pi}, K_{\rho}, c_{\pi}\text{ and }
c_{\rho }.$ For proving Lemma \ref{LIPSH-CONT}, we need the following result.

\begin{lemma}
For any given $U > L,r > 0,$ the function $\gamma(L,U;r,a,b)$ defined in %
\eqref{EXIT-PROB-EXPRESSION} is Lipschitz continuous with respect to the
variables $a \text{ and } b;$ that is, 
\begin{align*}
|\gamma(L,U;r,a_{1},\cdot)-\gamma(L,U;r,a_{2},\cdot)| < K(L,U,r)|a_{1}-a_{2}|
\\
|\gamma(L,U;r,\cdot,b_{1})-\gamma(L,U;r,\cdot,b_{2})| <
K(L,U,r)|b_{1}-b_{2}|,
\end{align*}
for all $a_{1},a_{2},b_{1},b_{2} \in(L,U).$ The Lipschitz constant $K(L,U,r)$
is given by 
\begin{equation*}
K(L,U,r) := \sum_{j \geq1} K_{j} = \frac{8(U-L)}{r} \sum_{j \geq1}
j\exp\left( -\frac{2}{r}(U-L)^{2}(j-1)^{2}\right) .
\end{equation*}
\label{LIPSCH-GAMMA}
\end{lemma}

\begin{proof}
Let $\gamma_{n}(a,b) = 1-\sum_{j=1}^{n}(\sigma_{j}-\tau_{j}).$ Since $a,b$
take values in $(L,U),$ it is easily checked that for all $j \geq1$ both 
\begin{align*}
\left| \frac{d}{d a}(\sigma_{j}-\tau_{j})\right| < K_{j}, \text{ and }
\left| \frac{d}{d b}(\sigma_{j}-\tau_{j})\right| < K_{j},
\end{align*}
where 
\begin{equation*}
K_{j}:=\frac{8(U-L)j}{r}\exp\left( -\frac{2}{r}(U-L)^{2}(j-1)^{2}\right) .
\end{equation*}
Then it is immediate that for all $n,$ 
\begin{equation*}
\left| \frac{d \gamma_{n}(a,b)}{d a}\right| < \sum_{j=1}^{\infty }K_{j} 
\text{ and } \left| \frac{d \gamma_{n}(a,b)}{d b}\right| <
\sum_{j=1}^{\infty}K_{j}.
\end{equation*}
As a consequence, we use the following elementary properties of Lipschitz
continuity to establish the Lipschitz continuity of $\gamma(L,U;r,a,b)$ with
respect to variables $a$ and $b$:

\begin{enumerate}
\item[1)] If a differentiable function $f(\cdot )$ on a convex domain
  is such that its first derivative $|f^{\prime }(x)|<K$ for some
  constant $K,$ then the function $f\left( \cdot \right) $ is
  Lipschitz continuous with Lipschitz constant at most $K.$

\item[2)] If a sequence of Lipschitz functions $f_{n}(\cdot)$ all
  having Lipschitz constant bounded by $K$ converge uniformly to
  $f(\cdot),$ then $%
  f(\cdot)$ is also Lipschitz continuous with Lipschitz constant at
  most $K.$
\end{enumerate}

Since $\gamma_{n}(a,b)$ converge uniformly to $\gamma(a,b)$ for $a,b
\in(L,U),$ it follows immediately from the above two facts that $\gamma
(L,U;r,a,b)$ is Lipschitz continuous with Lipschitz constant at most 
\begin{equation*}
K(L,U,r) := \sum_{j \geq1} K_{j} = \frac{8(U-L)}{r} \sum_{j \geq1}
j\exp\left( -\frac{2}{r}(U-L)^{2}(j-1)^{2}\right).
\end{equation*}
\end{proof}

\paragraph{Proof of Lemma \protect\ref{LIPSH-CONT}.}

For all $x \in(L^{\downarrow},U^{\uparrow}),$ 
\begin{equation}
\pi(x) = \exp\left( -\frac{1}{2}\left( x- \frac{s}{l}v \right) ^{2}\left. 
\frac{}{}\right/ \left( \frac{s(l-s)}{l}\right) \right) \leq1 =: c_{\pi}.
\label{CPI-BND}
\end{equation}
The Lipschitz continuity of $\pi(\cdot)$ follows from the boundedness of its
first derivative $\pi^{\prime}(\cdot)$ on the convex domain $(L,U):$ for all 
$x \in(L,U),$ 
\begin{equation}
\left| \frac{{d} \pi(x)}{{d} x}\right| \leq\frac {|xl-sv|}{s(l-s)} < \frac{%
\max\{|Ul - sv|, |Ll- sv|\}}{s(l-s)}=: K_{\pi}.  \label{LIPSCH-PI}
\end{equation}
To prove the Lipschitz continuity of $\rho(\cdot),$ we first note the
boundedness of $\gamma(L,U;r,\cdot,\cdot):$ Simple substitution will yield
that $\gamma(L,U;r,a,b) = 0$ whenever either $a$ or $b$ equals one of $L,U.$
Then due to the Lipschitz continuity of $\gamma(L,U;r,\cdot,\cdot)$ from
Lemma \ref{LIPSCH-GAMMA}, we have that 
\begin{equation}
|\gamma(L,U;r,a,b)| \leq K(L,U,r)(U-L).  \label{GAMMA-BOUND}
\end{equation}
Now consider the first term $\gamma_{1}(x) \gamma_{2}(x)$ in %
\eqref{RHO-GAMMA-EXP}:

\begin{enumerate}
\item[1)] Because of \eqref{GAMMA-BOUND}, $|\gamma_{1}(\cdot)|$ and
  $|\gamma _{2}(\cdot)|$ are bounded by $K(L^{\downarrow},
  U^{\uparrow}, s)(U^{\uparrow }-L^{\downarrow})$ and
  $K(L^{\downarrow}, U^{\uparrow}, l-s)(U^{\uparrow
  }-L^{\downarrow}),$ respectively, in the interval $x
  \in(L^{\downarrow },U^{\uparrow}).$

\item[2)] From Lemma \ref{LIPSCH-GAMMA}, we have that
  $\gamma_{1}(\cdot)$ and $%
  \gamma_{2}(\cdot)$ are Lipschitz continuous (with respect to the
  variable $x$%
  ) with Lipschitz constants at most $K(L^{\downarrow}, U^{\uparrow},
  s)$ and $%
  K(L^{\downarrow}, U^{\uparrow}, l-s)$ respectively.
\end{enumerate}

From the above two observations, we conclude that
$\gamma_{1}(\cdot)\gamma _{2}(\cdot)$ is Lipschitz continuous with
respect to $x$ with Lipschitz constant at most
\begin{equation*}
K_{1,2}:=2K(L^{\downarrow},U^{\uparrow},s)K(L^{\downarrow},U^{\uparrow
},l-s)(U^{\uparrow}-L^{\downarrow}).
\end{equation*}
This is because if $f,g$ are Lipschitz continuous with respective Lipschitz
constants $K_{f}$ and $K_{g}$ and absolute bounds $C_{f}$ and $C_{g},$ then $%
fg$ is Lipschitz continuous with Lipschitz constant at most $%
C_{f}K_{g}+C_{g}K_{f}.$ Using the same reasoning, the Lipschitz constants of
other terms in \eqref{RHO-GAMMA-EXP}, namely $\gamma_{3}(\cdot)\gamma_{4}(%
\cdot ),\gamma_{5}(\cdot)\gamma_{6}(\cdot)$ and $\gamma_{7}(\cdot)%
\gamma_{8}(\cdot)$ are at most 
\begin{align*}
K_{3,4} & :=2K(L^{\uparrow},U^{\uparrow},s)K(L^{\uparrow},U^{\uparrow
},l-s)(U^{\uparrow}-L^{\uparrow}), \\
K_{5,6} & :=2K(L^{\downarrow},U^{\downarrow},s)K(L^{\downarrow
},U^{\downarrow},l-s)(U^{\downarrow}-L^{\downarrow}),\text{ and } \\
<K_{7,8} & :=2K(L^{\uparrow},U^{\downarrow},s)K(L^{\uparrow},U^{\downarrow
},l-s)(U^{\downarrow}-L^{\uparrow})
\end{align*}
respectively. Therefore, $\rho(x)$ is Lipschitz continuous with Lipschitz
constant $K_{\rho}$ given by, 
\begin{equation*}
K_{\rho}:=K_{1,2}+K_{3,4}+K_{5,6}+K_{7,8}.
\end{equation*}
Since $\rho(x)=0$ whenever $x$ takes either $L^{\downarrow}$ or $U^{\uparrow
},$ using Lipschitz continuity of $\rho$ we reason that, 
\begin{equation*}
|\rho(x)|\leq K_{\rho}(U^{\uparrow}-L^{\downarrow})=:c_{\rho}.
\end{equation*}
This along with \eqref{CPI-BND} and \eqref{LIPSCH-PI} proves the claim.
\hfill$\Box$

\begin{lemma}
  Recall the definition
  $\bar{Z} := 1 + \sup_{t \in [0,1]}W(t) - \inf_{t \in [0,1]} W(t),$
  where $(W(t): t \in [0,1])$ is a standard Brownian motion. If 
  $(\bar{Z}_i: i =1,\ldots,k)$ are $k$ independent copies of
  $\bar{Z},$ then
  \begin{align*}
    E \left[\max_{i=1,\ldots,k} \bar{Z}_i \right] = \Theta \left(\sqrt{\log
    k} \right), \quad \text{ as } k \rightarrow \infty.  
  \end{align*}
\label{Lem-exp-max-BM}
\end{lemma}
\begin{proof}
We first observe that 
\begin{align*}
  \max_{i=1,\ldots,k} \bar{Z}_i 
&= \max_{i=1,\ldots,k} \left( 1 + \sup_{t \in [0,1]}W(t) - \inf_{t \in
  [0,1]} W(t) \right)\\ 
 &\leq 1 + \max_{i=1,\ldots,k}  \sup_{t \in [0,1]}W(t) +
   \max_{i=1,\ldots,k} \left( -\inf_{t \in  [0,1]} W(t)\right) 
\end{align*}
As $\sup_{t \in [0,1]}W(t) \overset{D}{=} -\inf_{t \in [0,1]}W(t),$ we
have,
\begin{align*}
  1 + \max_{i=1,\ldots,k}  \sup_{t \in [0,1]}W(t) \overset{D}{\leq} \max_{i=1,\ldots,k} \bar{Z}_i 
  \overset{D}{\leq} 1  + 2\max_{i=1,\ldots,k}  \sup_{t \in [0,1]}W(t), 
\end{align*}
where the notation $X \overset{D}{\leq} Y$ denotes that $X$ is
stochastically upper bounded by $Y.$ Further, as $\sup_{t \in
  [0,1]}W(t) \overset{D}{=} \vert Z \vert$ when $Z$ follows standard
normal distribution, we have 
\begin{align}
  E\left[\max_{i=1,\ldots,k} \bar{Z}_i\right] = \Theta \left( E\left[
  \max_{i=1,\ldots,k}  \sup_{t \in [0,1]}W(t) \right]\right) = \Theta \left(E\left[
  \max_{i=1,\ldots,k} \vert Z_i \vert \right]\right),
\label{Inter-L3}
\end{align}
as $k \rightarrow \infty.$ Here, $(Z_i: i = 1,\ldots,k)$ are simply
$k$ independent copies of a standard normal variable. Next, if we
denote the positive and negative parts of $Z_i$ as
$Z_i^+ = \max\{Z_i,0\}$ and $Z_i^- = -\min\{Z_i,0\},$ then
$\vert Z_i \vert = Z_i^+ + Z_i^{-}.$ Further, as
\[\max_{i=1,\ldots,k} Z_i^+ \leq \max_{i = 1,\ldots,k} \vert Z_i \vert \leq
\max_{i=1,\ldots,k} Z_i^+ + \max_{i=1,\ldots,k} Z_i^-,\] 
it follows from (\ref{Inter-L3}) that 
\begin{align*}
  E\left[\max_{i=1,\ldots,k} \bar{Z}_i\right]  = \Theta \left(
 E\left[ \max_{i=1,\ldots,k} Z_i^+ \right]\right) = \Theta \left(
 E\left[ M^+ \right]\right),
\end{align*}
where $M^+$ is the positive part of $M := \max_{i=1,\ldots,k}Z_i.$
Since $E[M] = \Theta(\sqrt{\log k})$ and
$E[M^-] = \int_0^\infty (P(Z < -u))^k du \rightarrow 0$ as
$k \rightarrow \infty,$ we obtain
$E[M^+] = E[M] + E[M^-] = \Theta(\sqrt{\log k}),$ thus proving the
claim.

\end{proof}

\bibliographystyle{abbrv}
\bibliography{RBM_Exact,reference}

\end{document}